\documentclass[preprint,12pt]{article}
\usepackage{amsmath,amssymb,amsthm}
\usepackage{graphicx}
\usepackage{subfigure}
\usepackage{color}
\usepackage{amssymb}
\usepackage{setspace}

\newtheorem{thm}{Theorem}[section]

 \newtheorem{lem}{Lemma}[section]
 \newtheorem{prop}{Proposition}[section]
 \newtheorem{defn}{Definition}[section]
\newtheorem{rem}{Remark}[section]

\def\ddj{\dot \Delta_j}

\def\tilde{\widetilde}
\def\hat{\widehat}

\def \zr1{$z_{R,1}$}
\def \zr2{z_{R,2}}
\def \zi1{z_{I,1}}
\def \zi2{z_{I,2}}

\def\Scho{{Schr$\ddot{o}$dinger}}

\def\v{{\tilde v}}
\def\u{{\tilde u}}

\newcommand\R{\mathbb{R}}

\begin{document}
\title{Global existence of Euler-Korteweg equations with the non-monotone pressure}
\author{Zihao Song$^{\mathrm{a},\mathrm{b}}$}
\date{\it \small $^a$ Research Institute for Mathematical Sciences, Kyoto University,\\
  Kyoto 606-8501, Japan\\
 $^b$ School of Mathematics, Nanjing University of Aeronautics and Astronautics,\\
Jiangsu Province, Nanjing 211106, China\\
 E-mail: szh1995@nuaa.edu.cn}

\maketitle
\begin{abstract}
We are concerned with the global solution of the compressible Euler-Korteweg equations in $\mathbb{R}^{3}$. In the case of zero sound speed $P'(\rho^{\ast})=0$, it is found that the perturbation problem of irrotational fluids could be reformulated into a quasi-linear  Schr$\ddot{o}$dinger equation. Based on techniques of dispersive estimates and methods of normal form, we construct a class of global scattering solutions for 3D case.
\end{abstract}

\section{Introduction}\setcounter{equation}{0}
In this paper, we study the compressible Euler-Korteweg model, whose theory formulation of (\ref{1.1}) was first introduced by Van der Waals \cite{V}, Korteweg \cite{K} and the rigorous derivation of corresponding equations for the compressible Euler-Korteweg system reads as:
\begin{equation}
\left\{
\begin{array}{l}\partial_{t}\rho+\mathrm{div}(\rho u)=0,\\ [1mm]
 \partial_{t}(\rho u)+\mathrm{div}(\rho u \otimes u)+\nabla P(\rho)=\mathrm{div}K(\rho).\\[1mm]
 \end{array} \right.\label{1.1}
\end{equation}
Here, $\rho=\rho(t,x)\in \mathbb{R}_{+}$ and $u=u(t,x)\in \mathbb{R}^{3}$ are unknown functions on $[0,+\infty)\times \mathbb{R}^{3}$, which stands for the density
and velocity field of a fluid, respectively. We neglect the thermal fluctuation so that the pressure $P=P(\rho)$ reduces to a function of $\rho$ only. In the following, the Korteweg tensor is given by (see \cite{BDDJ})
$$\mathrm{div}K(\rho)=\rho\nabla\big(\kappa(\rho)\Delta\rho+\frac{1}{2}\kappa'(\rho)|\nabla\rho|^{2}\big),$$
where the capillarity coefficient $\kappa>0$ may depend on $\rho$ in general.

The initial condition of System (\ref{1.1}) is prescribed by
\begin{equation}\label{1.2}
\left(\rho,u\right)|_{t=0}=\left(\rho _{0}(x),u_{0}(x)\right),\  x\in \R^{3}.
\end{equation}
In this paper, we investigate \eqref{1.1}-\eqref{1.2}, where initial data tends to a constant equilibrium $(\rho^{\ast},0)$ with $\rho^{\ast}>0$ if $|x|\rightarrow\infty$.

The model aims to describe the dynamic of a non-dissipative liquid-vapor mixture, where the phase changes are described through the variations of the density. Due to its strong physical background, there is an extensive literature of the (EKP) system in the past decades. Benzoni-Gavage, Danchin and Descombes \cite{BDD0} addressed the local well-posedness in one dimension by the Lagrangian formulation. In their later works \cite{BDD}, the local well-posedness theory and blow-up criterion for the multidimensional Euler-Korteweg system in Sobolev framework were established. For global results in physical dimensions, Antonelli and Marcati \cite{AP1,AP2} gave positive results concerning weak solutions with some particular case  $\kappa(\rho)=\frac{\kappa_{1}}{\rho}$, which corresponds to the so-called quantum fluids. Moreover, special solutions such as travelling wave solutions were constructed in Benzoni-Gavage \cite{B}.

For perturbation solutions near the constant state, it is noticed that if the fluid is irrotational, the linearized system of (\ref{1.1}) reflects a dispersive structure close to the well-known \Scho\, equation (NLS) reads as
\begin{equation}\label{Scho model}
i\partial_{t}\omega+\Delta\omega=\mathcal{N}(\omega).
\end{equation}
Fruitful mathematical results were established for (\ref{Scho model}) with various approaches in dispersive equations, such as Strichartz estimates \cite{CW,KP,Sh,St}, vector fields \cite{HN,GH} and normal forms \cite{C,GMS1,GMS2,W}. The appearance of dispersive structure for (\ref{1.1})-(\ref{1.2}) makes it possible to study small data/global problems near the equilibrium. In Audiard and Haspot \cite{AH,AH2}, authors studied the case when pressure is monotonically increasing on the constant equilibrium, i.e. $P'(\rho^*)>0$ where the perturbation equation could be reformulated into a dispersive system which looks alike to the Gross-Pitaevskii equation. Based on the normal form method applied in Gustafson, Nakanishi and Tsai \cite{GNT1,GNT2}, they established the global-in-time existence and uniqueness of small perturbation solutions for irrotational fluids in $d\geq3$ through delicate analysis of space time resonances.

On the other hand, owing to the fact that the Korteweg system was deduced by using Van der Waals potential, there exists non-monotone pressures due to the phase transition, see for example Kobayashi and Tsuda \cite{KT} for details. Therefore, it is also important to investigate the more physical cases of $P'(\rho^{\ast})=0$ (zero sound speed) and $P'(\rho^{\ast})<0$ in the mathematical analysis of \eqref{1.1}-\eqref{1.2}. In viscous cases, the linearized system shares parabolic mechanics thanks to the three order term of density, positive results concern global well-posedness were established \cite{CK,KT,MS,SX}.

However, there are few results concern global existence for inviscous Korteweg flows (Euler-Korteweg system) with vanishing sound speed at the constant equilibrium. In the present paper we investigate the zero sound speed case, where it is observed the perturbation system of \eqref{1.1} would be reformulated into a quasi-linear \Scho\, equation with quadratic nonlinearities. We would construct a class of global scattering solutions in three dimensions, which is actually regarded as a final data problem.

\subsection{Main results and strategy}

We state the following main result, which constructs a class of global unique solutions on $[0,\infty]$ which scatter asymptotically:
\begin{thm}\label{thm2.3}
Let $P'(\rho^{\ast})=0$. Assume $\varphi$ irrotational and $\varphi\in H^{3s+7}\cap\dot{H}^{-2}\cap B^{1}_{1,1}$, $x^{2}\varphi\in H^{1}$ for $s>\frac{5}{2}$. If $\varphi$ fulfills
\begin{eqnarray}
\|\varphi\|_{H^{3s+7}\cap\dot{H}^{-2}\cap B^1_{1,1}}+\|x^{2}\varphi\|_{H^{1}}\ll1,
\end{eqnarray}
then there exists a unique global-in-time solution $(\rho-\rho^*,u)$ for (\ref{1.1}) on $[0,\infty] $satisfying
$$(\nabla (\rho-\rho^*), u)\in \mathcal{C}([0,\infty],H^{s}\cap \dot{H}^{-1})$$
while
$$\rho-\rho^*=\delta\mathrm{Re}e^{it\Delta}\varphi+\mathrm{Re}\mathcal{Z};\,\,u=\mathrm{Im}e^{it\Delta}\nabla\varphi+\mathrm{Im}\nabla \mathcal{Z},$$
where $\delta$ is some positive constant while $\mathcal{Z}$ fulfills for any $t>0$
$$\|\mathrm{Re} \mathcal{Z}\|_{L^2}\lesssim t^{-\frac{3}{2}};\,\,\|\mathcal{Z}\|_{\dot{H}^{s}\cap\dot{H}^{1}}\lesssim t^{-\frac{3}{2}}.$$
\end{thm}

\begin{rem}
Theorem \ref{thm2.3} partially answers the global existence of Cauchy problem for (\ref{1.1})-(\ref{1.2}) in 3D case. Moreover, we prove there exists plenty of global solutions scattering to some free solutions in $L^2$ space where
\begin{eqnarray}
\|\rho-\rho^*-\delta\mathrm{Re}e^{it\Delta}\varphi\|_{L^2}\lesssim t^{-\frac{3}{2}};\,\|u-\mathrm{Im}e^{it\Delta}\nabla\varphi\|_{L^2}\lesssim t^{-\frac{3}{2}}.
\end{eqnarray}
The proof bases on studying a final data problem of (\ref{1.1}), which could be reformulated into a quasi-linear \Scho\,equation. We remark that it might be more difficult and interesting to study general $L^2$ initial data for Cauchy problem, as is pointed by Ikeda and Inui \cite{II}, the solution for 3D quadratic \Scho\, blows up in polynomial time with certain type nonlinearities.
\end{rem}

In the rest of this section, we illustrate the sketch of our proof. The linearization of (\ref{1.1}) and a very quick calculation for spectrum of (\ref{linearized}) shows that the linearized system evolutes as a quasi-linear {\Scho}  equation with quadratic and cubic nonlinearities. For the constructions of global solutions in 3D case, we shall actually study a final data problem for the following system:
\begin{equation}
\left\{
\begin{array}{l}\partial_{t}z-i\Delta z=\mathcal{N}^{2}(z)+\mathcal{N}^{3}(z),\\ [1mm]
\lim\limits_{t\rightarrow\infty}e^{-i\Delta t}z(t)=\varphi.\\[1mm]
 \end{array} \right.
\end{equation}
Due to the quasi-linearity, our strategy is to consider the energy method for solutions of some approximate equations with the final data on a sequence of time $T_{n}$. In order to overcome the derivative loss problem from those quasi-linear terms and non-transport term, we will introduce the gauge function and a modified density, which originally applied in \cite{BDD}.

Then the main task might be gaining enough decay to ensure the boundedness of $L^1_{[1,T_{n}]}L^\infty$ norm, which is essential for the evolution of Sobolev norm. Utilizing Duhamel formula, the final  data problem of (EK) model might be given as:
\begin{eqnarray}
z= z_{1}+i\int^{t}_{T_{n}}e^{i(t-s)\Delta}\big(\mathcal{N}^{2}(z)+\mathcal{N}^{3}(z)\big)ds,
\end{eqnarray}
where $z_{1}=e^{iT_{n}\Delta}\varphi$ represents the linear profile of the fixed function. The key ingredient of decay estimates is to establish the estimates for the following second approximation
$$z_{2}\triangleq \int^{t}_{T_{n}}e^{i(t-s)\Delta}\big(\mathcal{N}^{2}( z_{1})+\mathcal{N}^{3}( z_{1})\big)ds.$$
However, due to the dispersive estimates for \Scho\,operator, at the first glance, it is more direct to gain $t^{-\frac{1}{2}}$ decay from $z_{2}$. Hence one must improve the decay to ensure the integrability of time on $[1,T_{n}]$, which relies on a more delicate analysis on structure of quadratic nonlinearities. Indeed, one may roughly write the quadratic terms as
\begin{eqnarray}
\mathcal{N}^{2}(z)=\mathcal{N}(\nabla z)+\mathrm{Re}z\mathrm{Re}z+\mathrm{Re}z\Delta z.
\end{eqnarray}
An important observation is that those quadratic nonlinearities are either equipped with derivatives or they are all interactions of the real part of the variable $z$.

Therefore, for bounding the second approximation with derivatives, we shall apply the space-time resonance method introduced in \cite{GMS1,GMS2}. The idea is to estimate the space-time integral
$\int^{t}_{0}e^{is\Omega_{i}}B[\varphi,\varphi]ds$, $\Omega_{i}$ are equipped with Fourier symbols $\hat\Omega_{i}=|\xi|^2\pm|\eta|^2\pm|\xi-\eta|^2$, under the different type resonance sets:
\begin{itemize}
\item the time resonance set: $\mathcal{T}:=\{(\xi,\eta); \hat\Omega_{i}(\xi,\eta)=0\};$
\item the space resonance set: $\mathcal{S}:=\{(\xi,\eta); \nabla_{\eta}\hat\Omega_{i}(\xi,\eta)=0\};$
\item the time-space resonance set: $\mathcal{R}:=\{(\xi,\eta); \hat\Omega_{i}(\xi,\eta)=0\}\cap\{\nabla_{\eta}\hat\Omega_{i}(\xi,\eta)=0\}.$
\end{itemize}
We will integrate by parts either in time variable $s$ for non-time resonance set or in Fourier variable $\eta$ for non-space resonance set and we naturally the main contribution comes from space-time resonance $\mathcal{R}$. For \Scho\,\, equation, the main obstacle corresponds to the ``large" set of space-time resonance set $\xi=0$ for $|z|^2$ type nonlinearities. Precisely, integral by part on either variable would bring singularity near zero frequencies. However, thanks to the derivative in the second approximation, at least we are able to overcome those degeneracies in 3D.

On the other hand, for the real part of the second approximation, space-time resonance method might fail since degeneracies might be too strong to control. Our main idea is inspired by works of Hayashi and Naumkin \cite{HN0}. Indeed, we will utilize the following crucial equality
\begin{eqnarray*}
\mathrm{Re}\big(i\int^{t}_{T_{n}}e^{i(t-s)\Delta}|e^{is\Delta}\varphi|^2ds\big)=\mathrm{Im}\big(\int^{t}_{T_{n}}(e^{i(t-s)\Delta}-1)|e^{is\Delta}\varphi|^2ds\big)
\end{eqnarray*}
where the mean value theory allows us to the gain additional $t^{-1}$ decay in $L^2$ space.

Consequently, one is able to find the second approximation decay at rate $t^{-\frac{3}{2}}$ in the suitable space and this is enough to close the whole energy estimates. At last, we are able to establish the existence and uniqueness for sequence equation on $[1,T_{n}]$. Furthermore, one may extend the well-posedness to $t=0$ by the local theory established in \cite{BDD} and $t=\infty$ by a limit process on sequences $T_{n}$.

Finally, the rest of this paper unfolds as follows: in Section 2, we briefly recall the Littlewood-Paley theory and give the main bootstrap argument. Section 3 is devoted to present the second approximation of linear profile. Estimates for the main bootstrap would be given in Section 4. Finally we prove Theorem \ref{thm2.3} in Section 5.

\section{Preliminaries}\setcounter{equation}{0}

Throughout the paper, $C>0$ stands for a generic ``constant". For brevity, $f\lesssim g$ means that $f\leq Cg$. It will also be understood that $\|(f,g)\|_{X}=\|f\|_{X}+\|g\|_{X}$ for all $f,g\in X$. For $ 1\leq p\leq \infty$, we denote by $L^{p}=L^{p}(\mathbb{R}^{d})$ the usual Lebesgue space on $\mathbb{R}^{d}$ with the norm $\|\cdot\|_{L^{p}}$.

For $k\in\mathbb{N}$, we denote inhomogeneous Sobolev space by $W^{k,p}$ and homogeneous one with $\dot{W}^{k,p}$. If $p=2$, as usual we denote Bessel potential space $H^k$ and Riesz potential space $\dot{H}^k$ with $s\in\mathbb{R}$, where for $s$ a positive integer they coincide with the usual Sobolev spaces, i.e. $W^{k,p}=H^k$, $\dot{W}^{k,p}=\dot{H}^k$. The following product estimates concern Sobolev space would be frequently used through our paper:
$$\forall k\in\mathbb{N},\qquad\|fg\|_{W^{k,p}}\lesssim \|f\|_{L^{\infty}}\|g\|_{W^{k,p}}+\|f\|_{W^{k,p}}\|g\|_{L^{\infty}}.$$
Moreover, we have standardly for $F$ smooth satisfying $F(0)=0$, there holds for $u\in L^\infty\cap W^{k,p}$ such that
$$\|F(u)\|_{W^{k,p}}\lesssim C(\|u\|_{L^\infty})\|u\|_{W^{k,p}}.$$

\subsection{Littlewood-Paley theory and Besov spaces}
For convenience of reader, we would like to recall the Littlewood-Paley decomposition, Besov spaces and related analysis tools. The reader is referred to Chap.2 and Chap.3 of \cite{BCD} for more details. Let $\chi$ be a smooth function valued in $[0,1]$, such that $\chi$ is supported in the ball
$\mathbf{B}(0,\frac{4}{3})=\{\xi\in\mathbb{R}^{d}:|\xi|\leq\frac{4}{3}\}$. Set $\varphi(\xi)=\chi(\xi/2)-\chi(\xi)$. Then $\varphi$
is supported in the shell $\mathbf{C}(0,\frac{3}{4},\frac{8}{3})=\{\xi\in\mathbb{R}^{d}:\frac{3}{4}\leq|\xi|\leq\frac{8}{3}\}$ so that
$$\sum_{q\in\mathbb{Z}}\varphi(2^{-j}\xi)=1, \quad \forall\xi\in\mathbb{R}^{d}\backslash\{{0}\}.$$

For any tempered distribution $f\in\mathcal{S}'$, one can define the homogeneous dyadic blocks and homogeneous low-frequency cut-
off operators:
\begin{eqnarray*}
&&\dot{\Delta}_{j}f:=\varphi(2^{-j}D)f=\mathcal{F}^{-1}(\varphi(2^{-j}\xi)\mathcal{F}f), \quad j\in\mathbb{Z};
\end{eqnarray*}
\begin{eqnarray*}
&&\dot{S}_{j}f:=\chi(2^{-j}D)f=\mathcal{F}^{-1}(\chi(2^{-j}\xi)\mathcal{F}f), \quad j\in\mathbb{Z}.
\end{eqnarray*}

Denote by $\mathcal{S}'_{0}:=\mathcal{S'}/\mathcal{P}$ the tempered distributions modulo polynomials $\mathcal{P}$. As we known, the homogeneous
Besov spaces can be characterised in terms of the above spectral cut-off blocks.

\begin{defn}\label{defn2.1}
For $s\in \mathbb{R}$ and $1\leq p,r\leq \infty$, the homogeneous Besov spaces $\dot{B}^s_{p,r}$ are defined by
$$\dot{B}^s_{p,r}:=\Big\{f\in \mathcal{S}'_{0}:\|f\|_{\dot{B}^s_{p,r}}<\infty  \Big\} ,$$
where
\begin{equation*}
\|f\|_{\dot{B}^s_{p,r}}:=\Big(\sum_{q\in\mathbb{Z}}(2^{qs}\|\dot{\Delta}_qf\|_{L^{p}})^{r}\Big)^{1/r}
\end{equation*}
with the usual convention if $r=\infty$.
\end{defn}

We often use the following classical properties of Besov spaces (see \cite{BCD}):

$\bullet$ \ \emph{Bernstein inequality:}
Let $1\leq p\leq q\leq\infty$. Then for any $\beta,\gamma\in(\mathbb{N}^+)^d$, we have a constant $C$ independent of $f,j$ such that
$$supp\widehat f\subseteq\{A_{1}2^j\leq|\xi|\leq A_{2}2^j\} \Rightarrow C_{\mu}2^{j|\mu|}\|f\|_{L^p}\leq \sup_{\mu}\|\partial^\mu f\|_{L^p}\leq C_{\mu}2^{j|\mu|}\|f\|_{L^p},$$
$$supp\widehat f\subseteq\{A_{1}2^j\leq|\xi|\leq A_{2}2^j\} \Rightarrow \| f\|_{L^q}\leq C2^{-dj(\frac{1}{p}-\frac{1}{q})}\|f\|_{L^p}.$$

$\bullet$ \ \emph{Interpolation:}
The following inequality is satisfied for $1\leq p,r_{1},r_{2}, r\leq \infty, \sigma_{1}\neq \sigma_{2}$ and $\theta \in (0,1)$:
$$\|f\|_{\dot{B}_{p,r}^{\theta \sigma_{1}+(1-\theta )\sigma_{2}}}\lesssim \|f\| _{\dot{B}_{p,r_{1}}^{\sigma_{1}}}^{\theta} \|f\|_{\dot{B}_{p,r_2}^{\sigma_{2}}}^{1-\theta }$$
with $\frac{1}{r}=\frac{\theta}{r_{1}}+\frac{1-\theta}{r_{2}}$.

The product estimates in Besov spaces play a fundamental role. Let us first introduce Bony decomposition where for any distribution $f$ and $g$ such that
$$fg=T_{f}g+R(f,g)+T_{g}f$$
where paraproducts and remainders are defined as follow
$$T_{f}g\triangleq\sum_{j'\in \mathbb{Z}}\dot{S}_{j'-1}f\dot{\Delta}_{j'}g,\qquad R(f,g)\triangleq\sum_{j'\in \mathbb{Z}}\tilde{\dot{\Delta}}_{j'}f\dot{\Delta}_{j'}g.$$
where $\tilde{\dot{\Delta}}_{j'}=\sum\limits_{|j'-k|\leq1}{\dot{\Delta}}_{k}$.

Based on Bony decomposition, we immediately have the following Proposition concerns bounding bilinear terms in \eqref{linearized} (see \cite{BCD})
\begin{lem}\label{product}
Let $s_{1},s_{2}\in\mathbb{R}$ and $1\leq p,p_{1},q\leq\infty$. Then if $s_{1}<\frac{d}{p_{1}}$, then it holds
$$\|T_{a}b\|_{{\dot B}^{s_{1}+s_{2}-\frac{d}{p_{1}}}_{p,q}}\lesssim\|a\|_{{\dot B}^{s_{1}}_{p_{1},\infty}}\|b\|_{{\dot B}^{s_{2}}_{p,q}}.$$
Moreover, if $s_{1}+s_{2}> d\max\{0, \frac{2}{p}-1\}$, then it holds
$$\|R(a,b)\|_{{\dot B}^{s_{1}+s_{2}-\frac{d}{p_{1}}}_{p,q}}\lesssim\|a\|_{{\dot B}^{s_{1}}_{p_{1},\infty}}\|b\|_{{\dot B}^{s_{2}}_{p,q}}.$$
\end{lem}

\subsection{Linearization and bootstrap argument}

At the last subsection, we would like to present the linearization of (\ref{1.1}). For clarify, let the reference density $\rho^{*}=1$ such that $P'(1)=0$. Moreover, without loss of generality, we shall assume $\kappa(1)=1$. To symmetrize the system, we introduce the new variable $\ell=\mathcal{L}(\rho)$ where
\begin{eqnarray}\label{def ell}
\mathcal{L}(\rho)=\int^{\rho}_{1}\sqrt{\frac{\kappa(s)}{s}}ds
\end{eqnarray}
and the potential $\psi$ such that $\nabla\psi=u$, then the corresponding extend system of perturbation variable $(\ell,\psi)$ could be rewritten as
\begin{equation}\label{linearized}
\left\{
\begin{array}{l} \partial_{t}\ell+\Delta \psi=-\nabla\psi\cdot \nabla \ell-\tilde{a}(\ell)\Delta\psi,\\[1mm]
\partial_{t}\psi-\Delta \ell=\frac{1}{2}|\nabla \ell|^2-\frac{1}{2}|\nabla \psi|^2+\tilde{a}(\ell)\Delta\ell-\tilde{g}(\ell).\\ [1mm]
 \end{array} \right.
\end{equation}
Here, $\tilde{g}(\ell)=g\circ\mathcal{L}^{-1}(\ell)$ where $g$ is the bulk chemical potential of the fluid (see \cite{BDD}), by definition such that
$$\frac{dP}{d\rho}=\rho\frac{dg}{d\rho}$$
while $\tilde{a}(\ell)=a\circ\mathcal{L}^{-1}(\ell)-1$ where $a(\rho)=\sqrt{\rho\kappa(\rho)}$. $\tilde{a}(\ell)$ and $\tilde{g}(\ell)$ are assumed to be smooth and vanishing at zero point. Moreover, because of the zero sound speed condition, $\tilde{g}$ also satisfies $\tilde{g}'(0)=0$ which indicates $\tilde{g}(\ell)$ could be regarded as a nonlinear term. To prove the final data problem of (\ref{linearized}), we shall consider the following approximate equations concern $(\ell_{n},\psi_{n})$:
\begin{equation}\label{Appro}
\left\{
\begin{array}{l} \partial_{t}\ell_{n}+\Delta \psi_{n}=-\nabla\psi_{n}\cdot \nabla \ell_{n}-\tilde{a}(\ell_{n})\Delta\psi_{n},\\[1mm]
\partial_{t}\psi_{n}-\Delta \ell_{n}=\frac{1}{2}|\nabla \ell_{n}|^2-\frac{1}{2}|\nabla \psi_{n}|^2+\tilde{a}(\ell_{n})\Delta\ell_{n}-\tilde{g}(\ell_{n}),\\ [1mm]
(\ell_{n},\psi_{n})\big|_{t=T_{n}}=(\mathrm{Re} e^{iT_{n}\Delta}\varphi, \mathrm{Im} e^{iT_{n}\Delta}\varphi),\\[1mm]
 \end{array} \right.
\end{equation}
where $T_{n}\geq1$ is an increasing sequence while functions $\varphi$ are distributions. Define $z_{n}=\ell_{n}+i\psi_{n}$. (\ref{Appro}) implies $z_{n}$ satisfies
\begin{equation}\label{Scho2}
\left\{
\begin{array}{l}\partial_{t}z_{n}-i\Delta z_{n}=\mathcal{N}^{2}( z_{n})+\mathcal{N}^{3}( z_{n}),\\[1mm]
z_{n}|_{t=T_{n}}=z_{1}\\[1mm]
 \end{array} \right.
\end{equation}
where $z_{1}=e^{iT_{n}\Delta}\varphi$ and we further have the following integral equation by Duhamel theory:
\begin{eqnarray}\label{Duhamel2}
z_{n}= z_{1}+\int^{t}_{T_{n}}e^{i(t-s)\Delta}\big(\mathcal{N}^{2}( z_{n})+\mathcal{N}^{3}( z_{n})\big)ds.
\end{eqnarray}
where $\mathcal{N}^{2}( z)$ are quadratic nonlinear terms, could be precisely rewritten as
\begin{eqnarray}
\nonumber\mathcal{N}^{2}( z_{n})&=&-\nabla\mathrm{Im}z_{n}\cdot\nabla\mathrm{Re}z_{n}+i\big(\frac{1}{2}|\nabla\mathrm{Re}z_{n}|^{2}-\frac{1}{2}|\nabla\mathrm{Im}z_{n}|^{2}-\mathrm{Re}z_{n}\mathrm{Re}z_{n}
+\mathrm{Re}z_{n}\Delta z_{n}\big)\\
&=&\frac{i}{4}\big(2(\nabla z_{n})^{2}-z_{n}^{2}-\bar{z_{n}}^{2}-2|z_{n}|^{2}+2z_{n}\Delta z_{n}+2\bar{z_{n}}\Delta z_{n}\big)
\end{eqnarray}
and $\mathcal{N}^{3}( z)$ contains those cubic terms. Also, we define the second approximation
$$z_{2}\triangleq\int^{t}_{T_{n}}e^{i(t-s)\Delta}\big(\mathcal{N}^{2}( z_{1})+\mathcal{N}^{3}( z_{1})\big)ds$$
and $z_{2}=z_{2,1}+z_{2,2}$ where
$$z_{2,2}=-\frac{i}{2}\int^{t}_{T_{n}}e^{i(t-s)\Delta}(|z_{1}|^2+|\nabla z_{1}|^2)ds$$
while $z_{2,1}$ contains the rest of nonlinearities.

We end up this section with introducing our main bootstrap argument. We shall consider the following norm
$$\mathcal{Z}(t)\triangleq\sup_{t\in[1,T_{n}]}t^{\frac{3}{2}}\| z- z_{1}- z_{2}\|_{H^{s+1}}.$$
Our main purpose in this section is to prove the following Lemma concerns priori estimates for $\mathcal{Z}(t)$.
\begin{lem}\label{final priori2}
Assume $\varphi$ satisfies $\varphi\in H^{3s+7}\cap\dot{H}^{-2}\cap B^{1}_{1,1}$, $x^{2}\varphi\in H^{1}$ and define
\begin{eqnarray}\label{varepsilon}
\|\varphi\|_{H^{3s+7}\cap\dot{H}^{-2}\cap B^{1}_{1,1}}+\|x^{2}\varphi\|_{H^{1}}=\varepsilon_{0}.
\end{eqnarray}
Suppose that $(\ell_{n},\psi_{n})$ is the solution of (\ref{Appro}). Assume that
$$\mathcal{Z}_{\varepsilon_{0}}=C(\varepsilon_{0}+\varepsilon^{2}_{0}+\varepsilon^{3}_{0}),$$
then there holds for $s$ sufficient large such that
\begin{eqnarray}\label{Z priori}
\mathcal{Z}(t)\lesssim (1+\mathcal{Z}(t)+\mathcal{Z}_{\varepsilon_{0}})(\mathcal{Z}(t)+\mathcal{Z}_{\varepsilon_{0}})^{2}.
\end{eqnarray}
\end{lem}

\begin{rem}
The decay rates of $\| z- z_{1}- z_{2}\|_{H^{s+1}}$ are not sharp. In fact, one may improve it to $t^{-2}$ by requiring additional regularity for $\varphi$ in $L^1$ space. In order to relax the assumption on $\varphi$, we just obtain decay with $t^{-\frac{3}{2}}$ which is enough to ensure $L^1_{T}L^\infty$ norm and further scattering property.
\end{rem}

\section{Estimates for the second approximation}\setcounter{equation}{0}
In this section, we present the estimates concern estimates for $z_{2}$, which is the cornerstone for proof of Lemma \ref{final priori2}. We first give basic dispersive estimates concern \Scho semi-group $e^{it\Delta}$
\begin{lem}\label{dispersive}
For any distribution $g$, if $p\in[2,\infty]$, there holds
\begin{eqnarray}
\|e^{it\Delta}g\|_{L^p}\lesssim t^{\frac{3}{p}-\frac{3}{2}}\|g\|_{L^{p'}}.
\end{eqnarray}
We also have the following estimates for Lorenz space:
\begin{eqnarray}
\|e^{it\Delta}g\|_{L^{p,2}}\lesssim t^{\frac{3}{p}-\frac{3}{2}}\|g\|_{L^{p',2}}.
\end{eqnarray}
\end{lem}

The main goal of this section is to prove the following proposition:
\begin{prop}\label{w2}
If $\varphi$ satisfies $\varphi\in \dot{H}^{2}\cap\dot{H}^{-2}\cap B^{1}_{1,1}$, $x^{2}\varphi\in H^{1}$ and define
\begin{eqnarray}\label{varepsilon}
\|\varphi\|_{\dot{H}^{2}\cap\dot{H}^{-2}\cap B^{1}_{1,1}}+\|x^{2}\varphi\|_{H^{1}}=\mathcal{M},
\end{eqnarray}
then there holds for $t\geq1$ such that
\begin{eqnarray}\label{lplp}
\|\mathrm{Re} z_{2}\|_{L^2}\lesssim t^{-\frac{3}{2}}(\mathcal{M}^{2}+\mathcal{M}^{3}).
\end{eqnarray}
Moreover, if $\varphi$ additionally satisfies $\varphi\in \dot{H}^{2\alpha+1}$ with $\alpha\geq1$ and define
\begin{eqnarray}\label{varepsilon}
\|\varphi\|_{\dot{H}^{2\alpha+1}\cap\dot{H}^{\alpha-2}\cap B^{1}_{1,1}}+\|x^{2}\varphi\|_{H^{1}}=\tilde{\mathcal{M}},
\end{eqnarray}
then there holds for $t\geq1$ such that
\begin{eqnarray}\label{lplp1}
\| z_{2}\|_{\dot{H}^\alpha}\lesssim t^{-\frac{3}{2}}(\tilde{\mathcal{M}}^{2}+\tilde{\mathcal{M}}^{3}).
\end{eqnarray}
\end{prop}

\begin{proof}
We start with $L^2$ norm for $\mathrm{Re} z_{2}$. In fact $z_{2,1}$ could be formulated as
\begin{multline*}
z_{2,1}=\frac{i}{4}\int^{t}_{T_{n}}e^{i(t-s)\Delta}\big(-z_{1}z_{1}-\bar{z}_{1}\bar{z}_{1}+2\nabla z_{1}\cdot\nabla z_{1}+2z_{1}\Delta z_{1}
+2\mathrm{div}(\bar{z}_{1}\nabla z_{1})+\mathcal{N}^{3}( z_{1})\big)ds.
\end{multline*}
For the term $\bar{z}_{1}\bar{ z}_{1}$, we integral by parts for $s$ and have
\begin{eqnarray*}
&&\big\|\int^{t}_{T_{n}}e^{i(t-s)\Delta} \bar{z}_{1}\bar{ z}_{1}ds\big\|_{L^2}\\
&=&\big\|\int^{t}_{T_{n}}\int_{\mathbb{R}^3}e^{is\hat\Omega_{1}} \widehat{\bar{\varphi}_{1}}(\xi-\eta)\widehat{\bar{ \varphi}_{1}}(\eta)d\eta ds\big\|_{L^2}\\
&=&\big\|\int_{\mathbb{R}^3}\frac{1}{i\hat\Omega_{1}}
\big(\widehat{\bar{z}_{1}}(\xi-\eta,t)\widehat{\bar{z}_{1}}(\eta,t)-\widehat{\bar{z}_{1}}(\xi-\eta,T_{n})\widehat{\bar{z}_{1}}(\eta,T_{n})\big)d\eta \big\|_{L^2},
\end{eqnarray*}
therefore, by frequency decomposition, we have
\begin{multline}\label{443322}
\big\|\int^{t}_{T_{n}}e^{i(t-s)\Delta} \bar{z}_{1}\bar{ z}_{1}ds\big\|_{L^2}
\lesssim\|\bar{z}_{1}(t)\|_{L^\infty}\|\Lambda^{-2}\bar{ z}_{1}(t)\|_{L^2}+\|\bar{z}_{1}(T)\|_{L^\infty}\|\Lambda^{-2}\bar{ z}_{1}(T)\|_{L^2}\\
\lesssim (t^{-\frac{3}{2}}+T_{n}^{-\frac{3}{2}})\|\varphi\|_{L^1}\|\varphi\|_{\dot{H}^{-2}}\lesssim\! t^{-\frac{3}{2}}\mathcal{M}^{2}.
\end{multline}

As for term $z_{1}z_{1}$, by Bony decomposition, there holds
$$z_{1}z_{1}=2T_{z_{1}}z_{1}+R(z_{1},z_{1}).$$
Apparently for paraproduct, we could handle with
\begin{eqnarray*}
&&\big\|\int^{t}_{T_{n}}e^{i(t-s)\Delta} T_{z_{1}}{z_{1}}ds\big\|_{L^2}\\
&=&\big\|\sum_{j'\in \mathbb{Z}}\int^{t}_{T_{n}}\int_{\mathbb{R}^3}e^{is\hat\Omega_{3}}\widehat{\dot{S}_{j'-1}\varphi_{1}}(\xi-\eta)\widehat{\dot{\Delta}_{j'}\varphi_{1}}(\eta)d\eta ds\big\|_{L^2}\\
&\lesssim&\sum_{j'\in \mathbb{Z}}\big\|\frac{i}{2}\int^{t}_{T_{n}}\frac{1}{\tau}\int_{\mathbb{R}^3}\frac{\nabla_{\eta}\hat\Omega_{3}}{|\nabla_{\eta}\hat\Omega_{3}|^2}\nabla_{\eta}e^{is\Omega_{3}} \widehat{\dot{S}_{j'-1}\varphi_{1}}(\xi-\eta)\widehat{\dot{\Delta}_{j'}\varphi_{1}}(\eta)d\eta ds\big\|_{L^2}.
\end{eqnarray*}
Observe that $\nabla_{\eta}\hat\Omega_{3}=\xi-\eta-\eta\thicksim-2^{j'}$, we apply integral by parts and reach
\begin{multline*}
\big\|\int^{t}_{T_{n}}e^{i(t-s)\Delta} T_{z_{1}}{z_{1}}ds\big\|_{L^2}
\lesssim\|\Lambda^{-2}T_{z_{1}}{z_{1}}\|_{L^2}+\|\Lambda^{-1}T_{e^{it\Delta}x\varphi}{z_{1}}\|_{L^2}+\|\Lambda^{-1}T_{z_{1}}{e^{it\Delta}x\varphi}\|_{L^2}\\
\lesssim\|z_{1}\|_{L^\infty}\|z_{1}\|_{\dot{H}^{-2}}+\|e^{it\Delta}x\varphi\|_{L^4}\|z_{1}\|_{\dot{B}^{-1}_{4,2}}+\|z_{1}\|_{L^\infty}\|x\varphi\|_{\dot{H}^{-1}}.
\end{multline*}
Notice that
\begin{eqnarray}\label{x phi}
\|e^{i\tau\Delta}x\varphi\|_{L^4}\leq t^{-\frac{3}{4}}\|x\varphi\|_{L^\frac{4}{3}}\leq t^{-\frac{3}{4}}\|x^{2}\varphi\|^{\frac{1}{2}}_{L^2}\|\varphi\|^{\frac{1}{2}}_{L^1},
\end{eqnarray}
\begin{eqnarray}\label{x phi2}
\|z_{1}\|_{\dot{B}^{-1}_{4,2}}\leq t^{-\frac{3}{4}}\|\varphi\|_{\dot{B}^{-1}_{\frac{4}{3},2}}\leq t^{-\frac{3}{4}}\|\varphi\|^{\frac{1}{2}}_{\dot{H}^{-2}}\|\varphi\|^{\frac{1}{2}}_{\dot{B}^{0}_{1,1}},
\end{eqnarray}
while
\begin{multline}\label{x phi3}
\nonumber\|x\varphi\|_{\dot{H}^{-1}}\leq \big\{2^{-j}\|\ddj x\varphi\|_{L^2_{|x|\leq2^{-j}}}+2^{-j}\|\ddj x\varphi\|_{L^2_{|x|>2^{-j}}}\big\}_{l^2}\leq \|\varphi\|_{\dot{H}^{-2}}+\|x^{2}\varphi\|_{L^2},
\end{multline}
we immediately deduce
\begin{eqnarray}
\big\|\int^{t}_{T_{n}}e^{i(t-s)\Delta} T_{z_{1}}{z_{1}}ds\big\|_{L^2}\lesssim t^{-\frac{3}{2}}\mathcal{M}^{2}.
\end{eqnarray}

For the remainder term $R(z_{1},z_{1})$, we find actually $\hat\Omega_{3}$ is away from time resonant set, thus we just repeat calculations as in (\ref{443322}) and we eventually reach
\begin{eqnarray}
\big\|\int^{t}_{T_{n}}e^{i(t-s)\Delta} z_{1}z_{1}ds\big\|_{L^2}\lesssim\! t^{-\frac{3}{2}}\mathcal{M}^{2}.
\end{eqnarray}
Similarly we are able to bound terms $\nabla z_{1}\cdot\nabla z_{1}, z_{1}\Delta z_{1}$ provided additionally $\varphi\in \dot{H}^2\cap \dot{B}^{1}_{1,1}$ and $x^2\varphi\in \dot{H}^1$. As for $\mathrm{div}(\bar{z}_{1}\nabla z_{1})$, we shall utilize the derivative. Actually, the following equality always holds true:
\begin{eqnarray}\label{||}
\nabla(\bar{ f}g)=\frac{1}{it}(\bar{ f}e^{it\Delta}xe^{-it\Delta}g- g\overline{e^{it\Delta}xe^{-it\Delta}f}),
\end{eqnarray}
hence we deduce for $\mathrm{div}(\bar{z}_{1}\nabla z_{1})$
\begin{multline}
\big\|\int^{t}_{T_{n}}e^{i(t-s)\Delta} \mathrm{div}(\bar{z}_{1}\nabla z_{1})ds\big\|_{L^2}
\lesssim\int^{t}_{T_{n}}\frac{1}{s}\|\bar{ z}_{1}e^{it\Delta}x\nabla\varphi+\nabla z_{1}\overline{e^{it\Delta}x\varphi}\|_{L^2}ds\\
\lesssim\int^{t}_{T_{n}}\frac{1}{s}\big(\|z_{1}\|_{L^4}\|e^{it\Delta}x\varphi\|_{\dot{W}^{1,4}}+\| z_{1}\|_{\dot{W}^{1,\infty}}\|x\varphi\|_{L^2}\big)ds
\lesssim t^{-\frac{3}{2}}\mathcal{M}^{2}.
\end{multline}
In the end, for cubic terms $\mathcal{N}^{3}( z_{1})$, we have
\begin{eqnarray}
\big\|\int^{t}_{T_{n}}e^{i(t-s)\Delta} \mathcal{N}^{3}( z_{1})ds\big\|_{L^2}
\lesssim\int^{t}_{T_{n}}\|z_{1}\|^{2}_{L^\infty}\|\Delta z_{1}\|_{L^2}ds\
\lesssim t^{-2}\mathcal{M}^{3}
\end{eqnarray}
and we arrive at for $t\geq1$
\begin{eqnarray}\label{z1}
\|z_{2,1}\|_{L^2}\lesssim t^{-\frac{3}{2}}(\mathcal{M}^{2}+\mathcal{M}^{3}).
\end{eqnarray}

Now we turn to $\mathrm{Re}z_{2,2}$. Let us notice the following equality
\begin{eqnarray*}
\mathrm{Re}z_{2,2}&=&-\frac{1}{2}\mathrm{Im}\big(\int^{t}_{T_{n}}e^{i(t-s)\Delta}(|z_{1}|^2+|\nabla z_{1}|^2)ds\big)\\
&=&-\frac{1}{2}\mathrm{Im}\big(\int^{t}_{T_{n}}(e^{i(t-s)\Delta}-1)(|z_{1}|^2+|\nabla z_{1}|^2)ds\big).
\end{eqnarray*}
Therefore we have
\begin{eqnarray*}
\|\mathrm{Re}z_{2,2}\|_{L^2}
&\lesssim&\big\|\int^{t}_{T_{n}}(e^{i(s-\tau)\Delta}-1)(|z_{1}|^2+|\nabla z_{1}|^2)d\tau\big\|_{L^2}\\
&\lesssim&\int^{t}_{T_{n}}(s-\tau)\|\Delta(|z_{1}|^2+|\nabla z_{1}|^2)\|_{L^2}d\tau.
\end{eqnarray*}
Again by utilizing (\ref{||}), we have that
\begin{multline*}
\|\Delta|z_{1}|^2(\tau)\|_{L^2}
\lesssim\frac{1}{\tau^{2}}\|\big(z_{1}\overline{e^{i\tau\Delta}x^{2}\varphi}+|e^{i\tau\Delta}x\varphi|^2+\overline{z_{1}}e^{i\tau\Delta}x^{2}\varphi\big)(\tau)\|_{L^2}\\
\lesssim\frac{1}{\tau^{2}}\big(\|z_{1}\|_{L^\infty}\|x^{2}\varphi\|_{L^2}+\|e^{i\tau\Delta}x\varphi\|^{2}_{L^4}\big).
\end{multline*}
Keep in mind (\ref{x phi}), we arrive at
$$\|\Delta|z_{1}|^2(\tau)\|_{L^2}\lesssim\tau^{-\frac{7}{2}}\mathcal{M}^{2}.$$
Similar calculations on $|\nabla z_{1}|^2$ implies
\begin{eqnarray*}
\|\mathrm{Re}z_{2,2}\|_{L^2}&\lesssim& t^{-\frac{3}{2}}\mathcal{M}^{2}
\end{eqnarray*}
and we obtain the first inequality in (\ref{lplp}).

Secondly, we focus on estimating $z_{2}$ in regularity space $\dot{H}^\alpha$ for $\alpha\geq1$. We would repeat the computations for (\ref{z1}) since we have regularity to overcome those degeneracies, so it is enough to bound with those quasi-linear ones $z\Delta z$ and $\mathrm{div}(\bar{z}_{1}\nabla z_{1})$. For the first one, by Bony decomposition, we have
\begin{multline*}
\big\|\int^{t}_{T_{n}}e^{i(t-s)\Delta} T_{z_{1}}{\Delta z_{1}}ds\big\|_{\dot{H}^\alpha}
\lesssim\|T_{z_{1}}{z_{1}}\|_{\dot{H}^\alpha}+\|T_{e^{it\Delta}x\varphi}{\Lambda z_{1}}\|_{\dot{H}^\alpha}+\|T_{z_{1}}{\Lambda e^{it\Delta}x\varphi}\|_{\dot{H}^\alpha}\\
\lesssim\|z_{1}\|_{L^\infty}\|z_{1}\|_{\dot{H}^{\alpha}}+\|e^{it\Delta}x\varphi\|_{L^4}\|z_{1}\|_{\dot{W}^{\alpha+1,4}}+\|z_{1}\|_{L^\infty}\|x\varphi\|_{\dot{H}^{\alpha+1}}.
\end{multline*}
Since the following inequality holds true by Lemma \ref{dispersive}:
$$\|x\varphi\|_{\dot{H}^{\alpha+1}}\leq\|\varphi\|_{\dot{H}^{2\alpha+1}}+\|x^{2}\varphi\|_{\dot{H}^{1}};$$
$$\|e^{it\Delta}x\varphi\|_{L^4}\leq t^{-\frac{3}{4}}\|x\varphi\|_{L^{\frac{4}{3},2}}\leq t^{-\frac{3}{4}}\|\langle x^{2}\rangle\varphi\|_{L^2};$$
$$\|z_{1}\|_{\dot{W}^{\alpha+1,4}}\leq t^{-\frac{3}{4}}\|\varphi\|_{\dot{W}^{\alpha+1,\frac{4}{3}}}\leq t^{-\frac{3}{4}}(\|\varphi\|_{\dot{W}^{1,1}}+\|\varphi\|_{\dot{H}^{2\alpha+1}}),$$
we have
\begin{eqnarray*}
\big\|\int^{t}_{T_{n}}e^{i(t-s)\Delta} T_{z_{1}}{\Delta z_{1}}ds\big\|_{\dot{H}^\alpha}\lesssim t^{-\frac{3}{2}}\tilde{\mathcal{M}}^{2}.
\end{eqnarray*}
On the other hand, there holds
\begin{multline}
\big\|\int^{t}_{T_{n}}e^{i(t-s)\Delta} R(z_{1},\Delta z_{1})ds\big\|_{\dot{H}^\alpha}\lesssim\|z_{1}\|_{L^\infty}\|z_{1}\|_{\dot{H}^\alpha}
\lesssim t^{-\frac{3}{2}}\|\varphi\|_{L^{1}}\|\varphi\|_{\dot{H}^{\alpha}}\lesssim\! t^{-\frac{3}{2}}\tilde{\mathcal{M}}^{2}.
\end{multline}

Next we shall bound $\mathrm{div}(\bar{z}_{1}\nabla z_{1})$. There holds
\begin{multline}
\big\|\int^{t}_{T_{n}}e^{i(t-s)\Delta} \mathrm{div}(\bar{z}_{1}\nabla z_{1})ds\big\|_{\dot{H}^{\alpha}}
\lesssim\int^{t}_{T_{n}}\frac{1}{s}\|\bar{ z}_{1}e^{it\Delta}x\nabla\varphi+\nabla z_{1}\overline{e^{it\Delta}x\varphi}\|_{\dot{H}^{\alpha}}ds\\
\lesssim\int^{t}_{T_{n}}\frac{1}{s}\big(\| z_{1}\|_{W^{1,\infty}}\|x\varphi\|_{\dot{H}^{\alpha+1}}+\|e^{it\Delta}x\varphi\|_{W^{1,4}}\|z_{1}\|_{\dot{W}^{\alpha+1,4}}\big)ds
\lesssim t^{-\frac{3}{2}}\tilde{\mathcal{M}}^{2}.
\end{multline}
The other derivative term follows similar steps and we get to (\ref{lplp}).
\end{proof}

\section{Proof of Lemma \ref{final priori2}}\setcounter{equation}{0}
In this section, we are going to prove the main bootstrap argument, which based on energy estimate. For convenience of writing, we would hiding $n$ in the most calculation.

\subsection{$L^2$ estimates}

We begin with controlling $L^2$ norm. Define $\tilde{ z}= z- z_{1}- z_{2}$, then $\tilde{ z}$ satisfies
\begin{eqnarray}\label{tilde}
\tilde{ z}\nonumber&=&\int^{t}_{T_{n}}e^{i(t-s)\Delta}\big(\mathcal{N}^{2}( z)+\mathcal{N}^{3}( z)-\mathcal{N}^{2}( z_{1})-\mathcal{N}^{3}( z_{1})\big)ds.
\end{eqnarray}
For those nonlinearities, we shall take $\mathrm{Re}z\mathrm{Re}z$ as an example, there holds
\begin{multline}
i\int^{t}_{T_{n}}e^{i(t-s)\Delta}\big(\mathrm{Re}z\mathrm{Re}z-\mathrm{Re}z_{1}\mathrm{Re}z_{1}\big)ds=i\int^{t}_{T_{n}}e^{i(t-s)\Delta}\big(\mathrm{Re}\tilde{z}\mathrm{Re}\tilde{z}\\
+2\mathrm{Re}(z_{1}+z_{2})\mathrm{Re}\tilde{z}+2\mathrm{Re}z_{1}\mathrm{Re}z_{2}+\mathrm{Re}z_{2}\mathrm{Re}z_{2}\big)ds.
\end{multline}
Then taking $L^2$ norm yields for the former two terms:
\begin{multline}\label{iuiu}
\|\int^{t}_{T_{n}}e^{i(t-s)\Delta}\big(\mathrm{Re}\tilde{z}\mathrm{Re}\tilde{z}+2\mathrm{Re}(z_{1}+z_{2})\mathrm{Re}\tilde{z}\big)ds\|_{L^2}\\
\lesssim\int^{t}_{T_{n}}\|\mathrm{Re}\tilde{z}\|_{L^2}\big(\|\mathrm{Re}\tilde{z}\|_{\dot{H}^\frac{3}{2}}+\|\mathrm{Re}z_{1}+\mathrm{Re}z_{2}\|_{L^\infty}\big)ds\lesssim t^{-2}\mathcal{Z}^{2}_{\varepsilon_{0}}.
\end{multline}
On the other hand, for the latter ones, by Lemma \ref{w2}, we deduce
\begin{multline}
\|\int^{t}_{T_{n}}e^{i(t-s)\Delta}\big(2\mathrm{Re}z_{1}\mathrm{Re}z_{2}+\mathrm{Re}z_{2}\mathrm{Re}z_{2}\big)ds\|_{L^2}\\
\lesssim\int^{t}_{T_{n}}\|\mathrm{Re}z_{2}\|_{L^2}\|\mathrm{Re}z_{1}+\mathrm{Re}z_{2}\|_{L^\infty}ds\lesssim t^{-2}\mathcal{Z}^{2}_{\varepsilon_{0}}.
\end{multline}
Hence we get
\begin{eqnarray}\label{popokki}
\|i\int^{t}_{T_{n}}e^{i(t-s)\Delta}\big(\mathrm{Re}z\mathrm{Re}z-\mathrm{Re}z_{1}\mathrm{Re}z_{1}\big)ds\|_{L^2}\lesssim t^{-2}\mathcal{Z}^{2}_{\varepsilon_{0}}.
\end{eqnarray}
For those derivative nonlinearities, we only  treat $\mathrm{Re}z\Delta z$ where
\begin{multline}
i\int^{t}_{T_{n}}e^{i(t-s)\Delta}\big(\mathrm{Re}z\Delta z-\mathrm{Re}z_{1}\Delta z_{1}\big)ds=i\int^{t}_{T_{n}}e^{i(t-s)\Delta}\big(\mathrm{Re}\tilde{z}\Delta \tilde{z}
+\mathrm{Re}\tilde{z}\Delta (z_{1}+z_{2})\\
+\mathrm{Re}(z_{1}+z_{2})\Delta\tilde{z}+\mathrm{Re}z_{1}\Delta z_{2}+\mathrm{Re}z_{2}\Delta z_{1}+\mathrm{Re}z_{2}\Delta z_{2}\big)ds.
\end{multline}
We bound nonlinearities by similarly as (\ref{iuiu})-(\ref{popokki}) except $\mathrm{Re}\tilde{z}\Delta z_{1}$ and $\mathrm{Re}z_{2}\Delta z_{1}$ where
\begin{multline*}
\|\int^{t}_{T_{n}}e^{i(t-s)\Delta}\mathrm{Re}z_{2}\Delta z_{1}ds\|_{L^2}\lesssim\int^{t}_{T_{n}}\|\mathrm{Re}z_{2}\|_{L^3}\|\Delta z_{1}\|_{L^6}ds\\
\lesssim\int^{t}_{T_{n}}s^{-1}\|\mathrm{Re}z_{2}\|^{\frac{1}{2}}_{L^2}\|\mathrm{Re}z_{2}\|^{\frac{1}{2}}_{\dot{H}^1}\|\varphi\|^{\frac{2}{3}}_{\dot{W}^{1,1}}\|\varphi\|^{\frac{1}{3}}_{\dot{H}^{4}}ds
\lesssim t^{-\frac{3}{2}}(\mathcal{Z}(t)+\mathcal{Z}_{\varepsilon_{0}})^{2}.
\end{multline*}
Hence we reach for $s>\frac{5}{2}$
\begin{eqnarray*}
\|\int^{t}_{T_{n}}e^{i(t-s)\Delta}\big(\mathcal{N}^{2}( z)-\mathcal{N}^{2}( z_{1})\big)ds\|_{L^2}\lesssim t^{-\frac{3}{2}}(\mathcal{Z}(t)+\mathcal{Z}_{\varepsilon_{0}})^{3}.
\end{eqnarray*}
Similar calculations also hold for cubic terms where we notice
\begin{eqnarray}\label{cubic composite}
\|z\|_{L^{\infty}}\lesssim \|\tilde{ z}\|_{\dot{H}^{\frac{3}{2}}}+\|z_{1}\|_{\dot{W}^{\frac{1}{2},6}}+\|z_{2}\|_{\dot{H}^{\frac{3}{2}}}
\lesssim \mathcal{Z}(t)+t^{-\frac{3}{2}}\mathcal{Z}_{\varepsilon_{0}}
\end{eqnarray}
and we have for $t\geq1$ such that
\begin{eqnarray*}
\|\int^{t}_{T_{n}}e^{i(t-s)\Delta}\big(\mathcal{N}^{3}( z)-\mathcal{N}^{3}( z_{1})\big)ds\|_{L^2}\lesssim t^{-\frac{3}{2}}(\mathcal{Z}(t)+\mathcal{Z}_{\varepsilon_{0}})^{3}.
\end{eqnarray*}

Consequently, we conclude with
\begin{eqnarray}\label{low order}
\|\tilde{ z}\|_{L^{2}}\lesssim t^{-\frac{3}{2}}(1+\mathcal{Z}(t)+\mathcal{Z}_{\varepsilon_{0}})(\mathcal{Z}(t)+\mathcal{Z}_{\varepsilon_{0}})^{2}.
\end{eqnarray}

\subsection{High order Sobolev estimates}

Then we consider the high order derivative estimates. In fact, if we define $v_{n}=\nabla\ell_{n}$ and further set $\tilde{v}_{n}=v_{n}-\nabla z_{R,1}-\nabla z_{R,2}$ and $\tilde{u}_{n}=u_{n}-\nabla z_{I,1}-\nabla z_{I,2}$ with $z_{R,i}=\mathrm{Re} z_{i}, z_{I,i}=\mathrm{Im} z_{i}$, then (\ref{Appro}) becomes
\begin{equation}\label{Energy error}
\left\{\begin{array}{l}
\partial_{t}\tilde{v}_{n}+u_{n}\cdot \nabla \tilde{v}_{n}+v_{n}\cdot \nabla \tilde{u}_{n}+\nabla(a(\rho)\mathrm{div} \tilde{u}_{n})=F_{1}(\tilde{v}_{n},\tilde{u}_{n},\ell_{n}),\\ [1mm]
\partial_{t}\tilde{u}_{n}+u_{n}\cdot \nabla \tilde{u}_{n}-v_{n}\cdot \nabla \tilde{v}_{n}-\nabla(a(\rho)\mathrm{div} \tilde{v}_{n})= F_{2}(\tilde{v}_{n},\tilde{u}_{n},\ell_{n}),\\[1mm]
\end{array} \right.
\end{equation}
where
\begin{eqnarray*}
F_{1}(\tilde{v}_{n}, \tilde{u}_{n},\ell_{n})=\tilde{u}_{n}\cdot \nabla^{2} (z_{R,1}+z_{R,2})+\nabla z_{I,1}\cdot \nabla^{2} z_{R,2}+\nabla z_{I,2}\cdot \nabla^{2} z_{R,1}\\
+\nabla z_{I,2}\cdot \nabla^{2} z_{R,2}+\tilde{v}_{n}\cdot \nabla^{2} (z_{I,1}+z_{I,2})+\nabla z_{R,1}\cdot \nabla^{2} z_{I,2}+\nabla z_{R,2}\cdot \nabla^{2} z_{I,1}\\
+\nabla z_{R,2}\cdot \nabla^{2} z_{I,2}+\nabla(\tilde{a}(\ell_{n})\Delta (z_{I,1}+z_{I,2}))-\nabla(\tilde{a}(z_{R,1})\Delta z_{I,1})
\end{eqnarray*}
and
\begin{eqnarray*}
F_{2}(\tilde{v}_{n}, \tilde{u}_{n},\ell_{n})=\tilde{u}_{n}\cdot \nabla^{2} (z_{I,1}+z_{I,2})+\nabla z_{I,1}\cdot \nabla^{2} z_{I,2}+\nabla z_{I,2}\cdot \nabla^{2} z_{I,1}\\
+\nabla z_{I,2}\cdot \nabla^{2} z_{I,2}-\tilde{v}_{n}\cdot \nabla^{2} (z_{R,1}+z_{R,2})-\nabla z_{R,1}\cdot \nabla^{2} z_{R,2}-\nabla z_{R,2}\cdot \nabla^{2} z_{R,1}\\
\qquad-\nabla z_{R,2}\cdot \nabla^{2} z_{R,2}-\nabla(\tilde{a}(\ell_{n})\Delta (z_{R,1}+z_{R,2}))+\nabla(\tilde{a}(z_{R,1})\Delta z_{R,1})\\
-\bar{g}(\ell_{n})\ell_{n}+\bar{g}(z_{R,1})z_{R,1}
\end{eqnarray*}
while $\bar{g}(\ell_{n})$ is the composite function vanishing at 0. Again, for convenience, we drop $n$ in the following estimates. Now by imposing derivatives $\Delta^{\gamma}$ and gauge function $\phi(\rho)$ under the $L^2$ inner product, we write above equation into
\begin{multline}
\partial_{t}\|\phi\Delta^{\gamma} \v\|^{2}_{L^2}+\int\limits\phi\Delta^{\gamma} (u\cdot \nabla \tilde{v}+v\cdot \nabla \tilde{u}+\nabla(a(\rho)\mathrm{div} \tilde{u}))\cdot\phi\Delta^{\gamma} \v dx\\
=\int\limits\phi\Delta^{\gamma} F_{1}\cdot\phi\Delta^{\gamma} \v dx,
\end{multline}
\begin{multline}
\partial_{t}\|\phi\Delta^{\gamma} \u\|^{2}_{L^2}+\int\limits\phi\Delta^{\gamma}(u\cdot \nabla \tilde{u}-v\cdot \nabla \tilde{v}-\nabla(a(\rho)\mathrm{div} \tilde{v}))\cdot\phi\Delta^{\gamma} \u dx\\
=\int\limits\phi\Delta^{\gamma} F_{2}\cdot\phi\Delta^{\gamma} \u dx.
\end{multline}

Notice that terms in $F_{i}$ don't contain any derivative for $\tilde{v}, \tilde{u}$, one could directly apply the estimates. We take the term $\tilde{a}(\ell)\Delta \ell$ for an example. We have
$$\tilde{a}(\ell_{n})=(1+\bar{a}(\ell_{n}))\ell_{n},\,\, \tilde{a}(z_{R,1})=(1+\bar{a}(z_{R,1}))z_{R,1}$$
and further
\begin{eqnarray*}
\ell_{n}\Delta (z_{I,1}+z_{I,2})-z_{R,1}\Delta\nabla z_{I,1}&=&\mathrm{Re}\tilde{z}\Delta\nabla (\mathrm{Im}z_{1}+\mathrm{Im}z_{2})+\mathrm{Re}z_{1}\Delta\nabla \mathrm{Im}z_{2}\\
&+&\mathrm{Re}z_{2}\Delta\nabla \mathrm{Im}z_{1}+\mathrm{Re}z_{2}\Delta\nabla \mathrm{Im}z_{2}.
\end{eqnarray*}
Obviously if we replace $2\gamma$ by $s$, for the first term there holds
\begin{eqnarray}
\|\mathrm{Re}\tilde{z}\Delta\nabla (\mathrm{Im}z_{1}+\mathrm{Im}z_{2})\|_{\dot{H}^{s}}\lesssim \|\tilde{z}\|_{W^{s,3}}\|\Delta\nabla z_{1}\|_{W^{s,6}}
+\|\tilde{z}\|_{W^{s,6}}\|\Delta\nabla z_{2}\|_{W^{s,3}}.
\end{eqnarray}
Clearly we have
$$\|\Delta\nabla z_{1}\|_{W^{s,6}}\leq s^{-1}\|\varphi\|_{W^{s+3,\frac{6}{5}}}
\leq s^{-1}\|\varphi\|^{\frac{2}{3}}_{W^{1,1}}\|\varphi\|^{\frac{1}{3}}_{H^{3s+7}}\leq s^{-1}\mathcal{Z}_{\varepsilon_{0}}$$
when inspired by taking $\alpha=s+\frac{7}{2}$ in Lemma \ref{w2} yields
$$\|\Delta\nabla z_{2}\|_{W^{s,3}}\leq\|z_{2}\|_{\dot{H}^{\frac{7}{2}}\cap\dot{H}^{s+\frac{7}{2}}}\leq t^{-\frac{3}{2}}\mathcal{Z}_{\varepsilon_{0}}.$$
Consequently we have
\begin{eqnarray}
\|\mathrm{Re}\tilde{z}\Delta\nabla (\mathrm{Im}z_{1}+\mathrm{Im}z_{2})(s)\|_{H^{s}}\lesssim s^{-\frac{5}{2}}\mathcal{Z}(s)\mathcal{Z}_{\varepsilon_{0}}.
\end{eqnarray}
$\mathrm{Re}z_{1}\Delta\nabla \mathrm{Im}z_{2}, \mathrm{Re}z_{2}\Delta\nabla \mathrm{Im}z_{1}$ and $\mathrm{Re}z_{2}\Delta\nabla \mathrm{Im}z_{2}$ shall be treated similarly and we arrive at
\begin{eqnarray}
\|\ell_{n}\Delta u_{n}-z_{R,1}\Delta\nabla z_{I,1}\|_{H^{s}}\lesssim s^{-\frac{5}{2}}(\mathcal{Z}(s)+\mathcal{Z}_{\varepsilon_{0}})^{2}.
\end{eqnarray}
Other terms enjoy similar calculations and observe that by Lemma \ref{w2}, there holds for $t\geq1$
 \begin{eqnarray}\label{11}
&&\int\limits\phi\Delta^{\gamma} F_{1}\cdot\phi\Delta^{\gamma} \v dx+\int\limits\phi\Delta^{\gamma} F_{2}\cdot\phi\Delta^{\gamma} \u dx\\
\nonumber&\lesssim& t^{-\frac{5}{2}}(1+\mathcal{Z}(t)+\mathcal{Z}_{\varepsilon_{0}})(\mathcal{Z}(t)+\mathcal{Z}_{\varepsilon_{0}})\|\phi\Delta^{\gamma} (\v,\u)\|_{L^2}.
\end{eqnarray}

For the remaining nonlinearities which involve derivative loss, we start with transport structure where by commutator estimates and integral by parts, there holds
\begin{multline}
\int\limits\phi\Delta^{\gamma} (u\cdot \nabla \tilde{v})\cdot\phi\Delta^{\gamma} \v dx=\int\limits[\phi\Delta^{\gamma}, u\cdot \nabla] \v\cdot\phi\Delta^{\gamma} \v dx
+\int\limits u\cdot \nabla(\phi\Delta^{\gamma} \v)\cdot\phi\Delta^{\gamma} \v dx\\
\lesssim\big(1+\|\ell\|_{L^\infty}\big)\|\nabla u\|_{L^\infty}\|(\v,\u)\|^{2}_{H^{2\gamma}}\lesssim t^{-\frac{5}{2}}(1+\mathcal{Z}(t)+\mathcal{Z}_{\varepsilon_{0}})(\mathcal{Z}(t)+\mathcal{Z}_{\varepsilon_{0}})\|\phi\Delta^{\gamma}\v\|_{L^2}.
\end{multline}
Similarly, we have
\begin{eqnarray}\label{1111}
\int\limits\phi\Delta^{\gamma} (u\cdot \nabla \u)\cdot\phi\Delta^{\gamma} \u dx\lesssim
t^{-\frac{5}{2}}(1+\mathcal{Z}(t)+\mathcal{Z}_{\varepsilon_{0}})(\mathcal{Z}(t)+\mathcal{Z}_{\varepsilon_{0}})\|\phi\Delta^{\gamma}\u\|_{L^2}.
\end{eqnarray}
For the quasi-linear ones and non-transport term, we first have
$$\phi\Delta^{\gamma} \nabla(a(\rho)\mathrm{div} \u)=\nabla\big( a(\rho)\phi\Delta^{\gamma}\mathrm{div}\u\big)+2\gamma\nabla a(\rho)\phi\Delta^{\gamma}\mathrm{div}\u-
\nabla\phi a(\rho)\Delta^{\gamma}\mathrm{div}\u+C_{1};$$
$$\phi\Delta^{\gamma} \nabla(a(\rho)\mathrm{div}  \v)=\nabla\big( a(\rho)\phi\Delta^{\gamma}\mathrm{div}\v\big)+2\gamma\nabla a(\rho)\phi\Delta^{\gamma}\mathrm{div}\v-
\nabla\phi a(\rho)\Delta^{\gamma}\mathrm{div}\v+C_{2}$$
where $C_{1}, C_{2}$ only contain terms with at most $2\gamma$ order derivatives for $\u,\v$. Hence integrating by parts leads to
\begin{multline*}
\int\limits\phi\Delta^{\gamma}\nabla(a(\rho)\mathrm{div}\u)\cdot\phi\Delta^{\gamma}\v dx=-\int\limits a(\rho)\phi\Delta^{\gamma}\mathrm{div}\u\phi\Delta^{\gamma}\mathrm{div}\v dx\\
+2\gamma\int\limits\nabla a(\rho)\phi \Delta^{\gamma}\mathrm{div}\u\cdot\phi\Delta^{\gamma}\v dx-2\nabla\phi a(\rho)\Delta^{\gamma}\mathrm{div}\u\phi\Delta^{\gamma}\v dx+\tilde{C}_{1}.
\end{multline*}
Now for the fixed $\gamma$, selecting $\phi$ satisfies
\begin{eqnarray}\label{gauge}
\frac{a(\rho)}{\rho}+2\gamma a'(\rho)-2a(\rho)\frac{\phi'}{\phi}=0,
\end{eqnarray}
then the fact $v=\frac{a(\rho)}{\rho}\nabla\rho$ leads to
$$\int\limits\big(\phi v
+2\gamma\nabla a(\rho)\phi -2\nabla\phi a(\rho)\big)\Delta^{\gamma}\mathrm{div} \u\phi\Delta^{\gamma} \v dx=0$$
we immediately have
\begin{eqnarray*}
&&\int\limits\phi\Delta^{\gamma}\nabla(a(\rho)\mathrm{div} \u)\cdot\phi\Delta^{\gamma}\v dx+\int\limits\phi\Delta^{\gamma}( z\cdot\nabla \u)\cdot\phi\Delta^{\gamma}\v dx\\
&=&-\int\limits a(\rho)\phi\Delta^{\gamma}\mathrm{div} \u\phi\Delta^{\gamma}\mathrm{div} \v dx+\tilde{C}_{1}.
\end{eqnarray*}
Also, by symmetrical analysis, we could reach
\begin{eqnarray*}
&&-\int\limits\phi\Delta^{\gamma}\nabla(a(\rho)\mathrm{div} \v)\cdot\phi\Delta^{\gamma} \u dx-\int\limits\phi\Delta^{\gamma}(v\cdot\nabla \v)\cdot\phi\Delta^{\gamma} \u dx\\
&=&\int\limits a(\rho)\phi\Delta^{\gamma}\mathrm{div} \v\phi\Delta^{\gamma}\mathrm{div} \u dx-\tilde{C}_{2}.
\end{eqnarray*}
Therefore, above computations imply
\begin{multline}\label{111}
\int\limits\phi\Delta^{\gamma} (v\cdot \nabla \tilde{u}+\nabla(a(\rho)\mathrm{div} \tilde{u}))\cdot\phi\Delta^{\gamma} \v dx+\int\limits\phi\Delta^{\gamma}(-v\cdot \nabla \tilde{v}-\nabla(a(\rho)\mathrm{div} \tilde{v}))\cdot\phi\Delta^{\gamma} \u dx\\
=\tilde{C}_{1}+\tilde{C}_{2}\lesssim t^{-\frac{5}{2}}(1+\mathcal{Z}(t)+\mathcal{Z}_{\varepsilon_{0}})(\mathcal{Z}(t)+\mathcal{Z}_{\varepsilon_{0}})\|\phi\Delta^{\gamma} (\u,\v)\|_{L^2}.
\end{multline}
Finally, plugging (\ref{11})-(\ref{1111}) with (\ref{111}), we are able to conclude
$$\partial_{t}\|\phi\Delta^{\gamma} (\u,\v)\|^{2}_{L^2}\lesssim t^{-\frac{5}{2}}(1+\mathcal{Z}(t)+\mathcal{Z}_{\varepsilon_{0}})(\mathcal{Z}(t)+\mathcal{Z}_{\varepsilon_{0}})\|(\v,\u)\|^{2}_{H^{2\gamma}}.$$
Consequently by integrating on time for $[T_{n},t]$, we obtain
\begin{eqnarray*}
\|(\tilde{v}_{n},\tilde{u}_{n})\|^{2}_{\dot{H}^{s}}\lesssim \int^{t}_{T_{n}}s^{-\frac{5}{2}}(1+\mathcal{Z}(s)+\mathcal{Z}_{\varepsilon_{0}})
(\mathcal{Z}(s)+\mathcal{Z}_{\varepsilon_{0}})^{2}\|(\tilde{v}_{n},\tilde{u}_{n})\|_{\dot{H}^{s}}ds,
\end{eqnarray*}
that is
\begin{eqnarray}\label{high order}
\|\tilde{ z}\|_{\dot{H}^{s+1}}&\lesssim&t^{-\frac{3}{2}}(1+\mathcal{Z}(t)+\mathcal{Z}_{\varepsilon_{0}})(\mathcal{Z}(t)+\mathcal{Z}_{\varepsilon_{0}})^{2}.
\end{eqnarray}
Therefore combining (\ref{low order}) with (\ref{high order}) and by taking supreme norm on $t\in[1,T_{n}]$, we arrive at for all $s>\frac{5}{2}$ and $\alpha=3s+7$
\begin{eqnarray}\label{Z}
\mathcal{Z}(t)&\lesssim&(1+\mathcal{Z}(t)+\mathcal{Z}_{\varepsilon_{0}})(\mathcal{Z}(t)+\mathcal{Z}_{\varepsilon_{0}})^{2}
\end{eqnarray}
and we conclude with (\ref{Z priori}).

\section{Proof of Theorem \ref{thm2.3}}

Let us finish the provements of global existence and uniqueness. The following local in time existence and blow-up criterion for irrotational fluids have been established in \cite{BDDJ}:
\begin{lem}\label{lem local}
Let $u_{0}$ irrotational. If $s>\frac{5}{2}$, then for the initial data $(\rho_{0}-\rho^{*},u_{0})\in H^s$, there exists a positive $T$ such that (\ref{1.1}) admits a unique solution
$(\rho,u)$ such that
\begin{equation}
(\rho-\rho^{*}, u)\in \mathcal{C}([0,T],H^{s+1}\times H^s)\cap\mathcal{C}^{1}([0,T],H^{s-1}\times H^{s-2}).
\end{equation}
\end{lem}

\begin{lem}\label{lem blow}
Let $u_{0}$ irrotational. If $s>\frac{5}{2}$, then the solution of (\ref{1.1}) $(\rho,u)$ on $[0,T)$ could continue beyond $T$ once the following conditions satisfy:
\begin{itemize}
\item $\rho([0,T)\times \mathbb{R}^3)\subset J\subset \mathbb{R}^+$ where J is compact.
\item $\int^{T}_{0}\|\Delta\rho(t)\|_{L^\infty}+\|\mathrm{div}u(t)\|_{L^\infty}dt$ is bounded.
\end{itemize}
\end{lem}

According to (\ref{Z priori}), a small enough $\varepsilon_{0}\ll1$ would ensure the uniform bound of $\mathcal{Z}(t)$ on $[1,T_{n}]$. We also state that the boundedness of $\mathcal{Z}(t)$ ensures that
\begin{multline}
\int^{T_{n}}_{1}\|\Delta\rho(t)\|_{L^\infty}+\|\mathrm{div}u(t)\|_{L^\infty}dt\lesssim\int^{T_{n}}_{1}\|\Delta z\|_{L^\infty}dt\\
\lesssim\int^{T_{n}}_{1}\|\tilde{z}\|_{\dot{H}^\frac{7}{2}}+\|z_{1}\|_{\dot{W}^{\frac{5}{2},6}}+\|z_{2}\|_{\dot{H}^\frac{7}{2}}dt\lesssim C_{\mathcal{Z}_{\varepsilon_{0}}}
\end{multline}
provided $s\geq\frac{5}{2}$. Therefore by local theory and blow-up criteria, we prove the existence and uniqueness on $[1,T_{n}]$.

We take the limit of $T_{n}$ where $n\rightarrow\infty$ and apply Lemma \ref{lem local} on $t=1$, then the solution is uniquely extended on $[0,+\infty]$. Moreover, the continuity of the solution is got by a Bona-Smith argument. As for the scattering property, we start with the density, by the definition of $\ell$ and Taylor extension, we have
$$\rho-1=\mathcal{L}^{-1}(\ell)-1=\delta\ell+\bar{\mathcal{L}}(\ell)\ell.$$
where $\delta={(\mathcal{L}^{-1})'}(0)$. Therefore, there holds
\begin{eqnarray}
\|\rho-1-\delta\mathrm{Re}e^{it\Delta}\varphi\|_{L^2}&\lesssim&\|\ell-\mathrm{Re}e^{it\Delta}\varphi\|_{L^2}+\|\ell\|_{L^2}\|\ell\|_{L^\infty}.
\end{eqnarray}
Keep in mind Lemma \ref{w2} that for $t>0$
\begin{eqnarray}
\|\ell-\mathrm{Re}e^{it\Delta}\varphi\|_{L^2}&\lesssim&\|\mathrm{Re} \tilde{z}\|_{L^2}+\|\mathrm{Re} z_{2}\|_{L^2}\lesssim t^{-\frac{3}{2}}
\end{eqnarray}
while
\begin{eqnarray}
\|\ell\|_{L^\infty}&\lesssim&\|\mathrm{Re} \tilde{z}\|_{\dot{H}^\frac{3}{2}}+\|\mathrm{Re} z_{1}\|_{\dot{W}^{\frac{1}{2},6}}+\|\mathrm{Re} z_{2}\|_{\dot{H}^\frac{3}{2}}\lesssim t^{-\frac{3}{2}}.
\end{eqnarray}
We could conclude with
\begin{eqnarray}
\|\rho-1-\delta\mathrm{Re}e^{it\Delta}\varphi\|_{L^2}&\lesssim&t^{-\frac{3}{2}}.
\end{eqnarray}
On the other hand, for velocity, we similarly have
\begin{eqnarray}
\|u-\mathrm{Im}e^{it\Delta}\nabla\varphi\|_{L^2}\lesssim\|\nabla(\psi-\mathrm{Im}e^{it\Delta}\varphi)\|_{L^2}&\lesssim&t^{-\frac{3}{2}}.
\end{eqnarray}

\noindent {\bf Acknowledgments:}\ \

The author sincerely thanks Professor Kenji Nakanishi for helpful suggestions and discussions in the course of this research.  The author is partially supported by ``Muxing" project of Nanjing University of Aeronautics and Astronautics.

\end{document}